\documentclass[reqno]{amsart}

\usepackage{amssymb}
\usepackage{graphicx}
\usepackage{psfrag}

\usepackage[T1]{fontenc}

\newtheorem{theorem}{Theorem}
\newtheorem{lemma}[theorem]{Lemma}
\newtheorem{conjecture}{Conjecture}

\theoremstyle{definition}
\newtheorem{definition}[theorem]{Definition}
\newtheorem{example}[theorem]{Example}

\theoremstyle{remark}

\newcommand\fix{{\rm{Fix}}}

\newcommand\failure{{\rm{Fail}}}
\newcommand\denom{{\rm{Denom}}}
\newcommand\rad{{\rm{Rad}}}
\newcommand\lcm{{\rm{lcm}}}
\newcommand\divides{{\mathchoice{\mathrel{\bigm|}}{\mathrel{\bigm|}}{\mathrel{|}}%
       {\mathrel{|}}}}
\newcommand\smalldivides{\mathrel{\kern-2pt\kern3.5pt|}}
\newcommand\notdivides{\mathrel{\kern-3pt\not\!\kern4.3pt\bigm|}}
\newcommand\smallnotdivides{\mathrel{\kern-2pt\not\!\kern3.5pt|}}

\def\({\left(}
\def\){\right)}

\def\PP{\mathbb P}

\def\C{\mathbb C}

\def\N{\mathbb N}

\def\S{\mathbb S}

\def\eqref#1{{{\rm(\ref{#1})}}}
\def\beginfig{\begin{figure}[!htbp]\begin{center}}
\def\endfig{\end{center}\end{figure}}
\renewcommand\ge{\geqslant}
\renewcommand\le{\leqslant}

\renewcommand\leq{\leqslant}
\renewcommand\subset{\subseteq}

\def\stirlingtwo#1#2{S^{(2)}(#1,#2)}
\def\stirlingone#1#2{S^{(1)}(#1,#2)}
\def\stirlingonesigned#1#2{S^{(1)}_{\pm}(#1,#2)}
\def\sequencetwo{{S}^{(2)}}
\def\sequenceone{{S}^{(1)}}

\begin{document}

\title{Stirling numbers and periodic points}


\author{Piotr Miska}
\address{Faculty of Mathematics and Computer Science, Institute of Mathematics, Jagiellonian University, {\L}ojasiewicza 6, 30--348, Krak{\'o}w, Poland.}
\email{piotr.miska@uj.edu.pl}

\author{Tom Ward}
\address{School of Mathematics, Newcastle University, Newcastle NE1 7RU, U.K.}
\email{tom.ward@newcastle.ac.uk}

\thanks{The results discussed here were
conjectured based on numerical experiments
by the second named author, leading the first named
author to discover an appropriate framework and
proofs. The research of the first author is supported by the grant of the Polish National Science Centre no. UMO-2019/34/E/ST1/00094}


\date{\today}

\begin{abstract}
We introduce the notion of `almost realizability',
an arithmetic generalization of `realizability' for
integer sequences, which is the property of counting periodic
points for some map.
We characterize the intersection between
the set of Stirling sequences (of both the first and the second kind)
and the set of almost realizable sequences.
\end{abstract}

\maketitle

\section{Introduction, definitions, and examples}

Denote the sets of non-negative integers and prime numbers by~$\N$ and~$\PP$, respectively.
Write~$\stirlingone{n}{k}$ for the (signless) Stirling
numbers of the first kind, defined for any~$n\ge1$ and~$0\le k\le n$
to be the number of permutations of~$\{1,\dots,n\}$
with exactly~$k$ cycles.
Write~$\stirlingonesigned{n}{k}$ for the (signed)
Stirling numbers of the first kind, defined by
the relation
\[
(x)_n=x(x-1)(x-2)\cdots(x-n+1)=\sum_{k=0}^{n}\stirlingonesigned{n}{k}x^k
\]
for all~$n\ge1$ and~$0\le k\le n$.
The two are related by~$(-1)^{n-k}\stirlingonesigned{n}{k}=\stirlingone{n}{k}$
for all~$n\ge1$ and~$0\le k\le n$.
Finally, let~$\stirlingtwo{n}{k}$ for~$n\ge1$ and~$1\le k\le n$
denote the Stirling numbers of the second kind,
so~$\stirlingtwo{n}{k}$ counts the number of ways to partition
a set comprising~$n$ elements
into~$k$ non-empty subsets.

\begin{definition}
For each~$k\ge1$ define sequences
\begin{align*}
\sequenceone_k
&=
\bigl(\stirlingone{n+k-1}{k}\bigr)_{n\ge1}
\intertext{and}
\sequencetwo_k
&=
\bigl(\stirlingtwo{n+k-1}{k}\bigr)_{n\ge1}.
\end{align*}
\end{definition}

The properties we wish to discuss depend entirely
on the exact offset~$(k-1)$ chosen in the index~$n$, so for definiteness we
mention two examples:
\begin{align}
\sequenceone_3&=(1,6,35,225,\dots)\label{stirlingoneat3sequence}
\intertext{and}
\sequencetwo_3&=(1,6,25,90,\dots).\label{stirlingtwoat3sequence}
\end{align}

A natural combinatorial property that
a sequence of non-negative integers may have
is that it counts periodic points for iterates
of a map. For our purposes a map can simply be
regarded as a permutation of~$\N$ with the
property that there are only finitely many cycles of
each length (though one may equally well ask that
the map be a~$C^{\infty}$ diffeomorphism of the
annulus by a result of Windsor~\cite{MR2422026}).

\begin{definition}\label{definitionRealizable}
An integer sequence~$A=(A_n)$
is called \emph{realizable} if either of the two
equivalent conditions holds:
\begin{enumerate}
\item[(a)] There is a map~$T\colon X\to X$
on a set~$X$
with the property that
\[
A_n=\fix_{(X,T)}(n)=\vert\{x\in X\mid T^nx=x\}\vert
\]
for all~$n\geqslant1$; or
\item[(b)] The sequence~$A$ satisfies both the
\emph{Dold condition}
\begin{equation}\label{equationDoldCondition}
n\divides(\mu*A)(n)=\sum_{d\smalldivides n}\mu({n}/{d})A_d
\end{equation}
for all~$n\ge1$
(from~\cite{MR724012})
and the \emph{sign condition}
\[
(\mu*A)(n)=\sum_{d\smalldivides n}\mu({n}/{d})A_d\geqslant0
\]
for all~$n\ge1$.
\end{enumerate}
\end{definition}

Here~$\mu$ denotes the classical M{\"o}bius function,
and~$*$ denotes Dirichlet convolution.
The equivalence between the two definitions is just
a consequence of the fact that the set of
points of period~$n$ under a map is the disjoint
union of the points living on closed orbits of
length~$d$ for~$d\divides n$, and each closed
orbit of length~$d$ has exactly~$d$ points on it.
We refer to work of Pakapongpun, Puri and Ward~\cite{MR2486259,MR3194906,MR1873399}
for more on this equivalence and some
of its functorial consequences.

\begin{example}\label{exampleNotRelisableAtn=2andk=3}
The sequences~$\sequenceone_3$ and~$\sequencetwo_3$
are not realizable. To see this, notice that the
Dold condition applied at~$n=2$ shows that
if the sequence~$A$ is realizable, then~$A_2-A_1$ must be
even, while~\eqref{stirlingoneat3sequence}
and~\eqref{stirlingtwoat3sequence} do not have
this property.
\end{example}

Every rational number~$x$ can be written in a unique way as a quotient of two coprime integers with positive denominator. Let us denote the denominator in this representation of~$x$ as~$\denom(x)$.
If additionally~$x\neq0$, then there exists a unique
sequence~$(\alpha_p)_{p\in\PP}$ of integers, where~$\alpha_p\neq 0$ only for finitely many primes~$p$, such that
\[
|x|=\prod_{p\in\PP} p^{\alpha_p}.
\]
We then define the~$p$-adic norm of~$x$
to be
\[
|x|_p=p^{-\alpha_p}.
\]

Ignoring the sign condition for the moment,
an integer sequence~$A$ fails to satisfy the
Dold condition if
and only if
\[
\denom\(\tfrac1n(\mu*A)(n)\)
=
\prod_{p\in\PP}\max\left\{1,\left\vert\tfrac{1}{n}(\mu*A)(n)\right\vert_p\right\}>1
\]
for some~$n\ge1$.
This gives a \emph{measure of failure}
for a sequence~$A$ to be realizable,
as follows.

\begin{definition}\label{definitionAlmostRealizable}
For a sequence~$A$ of non-negative integers, we write
\[
\failure(A)
=
\begin{cases}
\lcm
\(\{
\denom\((\tfrac1n\mu*A)(n)\)\mid n\ge1
\}\)&\mbox{if this is finite};\\
\infty&\mbox{if not}.
\end{cases}
\]
The sequence~$A$ is said to be \emph{almost realizable}
if~$\failure(A)<\infty$ and
it satisfies the sign condition.
\end{definition}

A realizable sequence~$A$ has~$\failure(A)=1$,
and if~$A$ is almost realizable, then
\[
\failure(A)\cdot A
=
(\failure(A)\cdot A_n)_{n\ge1}
\]
is realizable. That is, in such a case
the failure to be
realizable can be `repaired'
simply by multiplying by a single number.
The motivation for this (admittedly odd) notion comes
from a result of Moss and Ward, where it arose in an
unexpected setting.

\begin{example}[Moss and Ward~\cite{mw}]
The Fibonacci sequence~$F=(F_n)=(1,1,2,\dots)$ is
not almost realizable, but the Fibonacci sequence
sampled along the squares~$(F_{n^2})$ is almost
realizable with~$\failure ((F_{n^2}))=5$.
\end{example}

\section{Main results}

Our results establish when---and to what extent---the
Stirling numbers satisfy Definitions~\ref{definitionRealizable}
and~\ref{definitionAlmostRealizable}. In order to expose
the interaction between the arithmetic properties of
Stirling numbers and the properties of realizability and almost
realizability we move the proofs of the purely
combinatorial lemmas into Section~\ref{sectionLemmaProofs}.

\begin{theorem}[Stirling numbers of the first kind]\label{theoremfirstkind}
For~$k\ge1$ the sequence~$\sequenceone_k$ is
not almost realizable.
\end{theorem}

\begin{theorem}[Stirling numbers of the second kind]\label{theoremsecondkind}
For~$k\le2$ the sequence~$\sequencetwo_k$ is realizable.
For~$k\ge3$ the sequence~$\sequencetwo_k$ is not
realizable, but is almost realizable
with~$\failure\bigl(\sequencetwo_k\bigr)\divides(k-1)!$
for all~$k\ge1$.
\end{theorem}

\begin{proof}[Proof of Theorem~\ref{theoremsecondkind} for~$k\le2$.]
We have
\[
\sequencetwo_1=(1,1,1,1,\dots)
\]
which is realized by the identity map on a singleton.
For~$k=2$ we have
\[
\sequencetwo_2=(1,3,7,15,\dots)=(2^n-1).
\]
This is easily seen to be realizable, either by checking the
conditions or by noticing that if~$T$ is
the map~$z\mapsto z^2$ on~$\S^1=\{z\in\C\mid\vert z\vert=1\}$,
then~$\fix_{(X,T)}(n)=2^n-1$ for all~$n\ge1$.
\end{proof}

The non-trivial parts of the arguments below concern
proving that the Dold condition is---or is not---satisfied,
and to what extent.
The sequences we are concerned with grow very rapidly, so the
sign condition is not a concern, but for completeness
we include a simple argument
that verifies the
sign condition for~$\sequencetwo_k$.

\begin{lemma}\label{lemmaYash}
For~$k\ge1$ we have~$(\mu*\sequencetwo_k)(n)\ge0$ for all~$n\ge1$.
\end{lemma}

The proof of Theorem~\ref{theoremfirstkind} for~$k\ge3$
involves finding a prime index at which the Dold congruence fails---that is,
finding a witness to non-realizability like
the prime~$2$ in
Example~\ref{exampleNotRelisableAtn=2andk=3}.
In order to do this, we need some information on
the properties of~$\sequenceone$ modulo a prime.

\begin{lemma}\label{lemmaStirlingOneMainCombinatorialArgument}
For a prime~$p$ and~$k\in\N$ we have
\begin{equation}\label{equationStirlingOneValuesModulop1}
\stirlingone{p+k-1}{k}\equiv1\pmod{p}
\end{equation}
if and only if~$k\equiv j\pmod{p^2}$
for some~$j\in\{(p-2)p+1,(p-2)p+2,\dots,(p-1)p\}$.
For an odd prime~$p$ and~$k\in\N$ we have
\begin{equation}\label{equationStirlingOneValuesModulop2SignedCase}
\stirlingonesigned{p+k-1}{k}\equiv1\pmod{p}
\end{equation}
if and only if~$k\equiv j\pmod{p^2}$
for some~$j\in\{(p-2)p+1,(p-2)p+2,\dots,(p-1)p\}$.
\end{lemma}

\begin{proof}[Proof of Theorem~\ref{theoremfirstkind}.]
Let~$p$ be a prime and~$k\in\N$.
If~$k\le(p-2)p$, then we certainly have~$k\notin\{(p-2)p+1,\dots,(p-1)p\}\pmod{p^2}$.
By Lemma~\ref{lemmaStirlingOneMainCombinatorialArgument}
we have
\[
\stirlingone{p+k-1}{k}\not\equiv1\pmod{p}.
\]
Since~$\stirlingone{k}{k}=1$, we deduce that~$\bigl\vert\frac1p\bigl(\stirlingone{p+k-1}{k}-\stirlingone{k}{k}\bigr)\bigr\vert_p=p$,
so the set of denominators among the expressions~$\frac1n(\mu*\sequenceone)(n)$
contains infinitely many primes. Thus~$\sequenceone$ is not almost
realizable.

The same argument for the signed Stirling numbers is similar, and
shows that they fail the Dold congruence in the same strong way.
\end{proof}

The structure of the proof that~$\sequencetwo_k$ is not realizable for~$k\ge3$
is very similar to that
of Theorem~\ref{theoremfirstkind}, but the details differ.
To do this, once again some information on the properties of~$\sequencetwo_k$
modulo a prime is needed.

\begin{lemma}\label{lemmaS2CongruencesAndPeriodicity}
For a prime~$p$ and~$k\in\N$ we have
\begin{equation}\label{equationWhenIsS21ModuloP}
\stirlingtwo{p+k-1}{k}\equiv1\pmod{p}
\end{equation}
if and only if~$k\equiv j\pmod{p^2}$ for some~$j\in\{1,\dots,p\}$.
\end{lemma}

\begin{proof}[Proof of Theorem~\ref{theoremsecondkind} for~$k\ge3$: non-realizability.]
If~$A$ is a realizable sequence,
then the Dold condition~\eqref{equationDoldCondition}
requires~$A_p\equiv A_1\pmod{p}$ .
So it is enough to show that for~$k\ge3$ there is
a prime~$p$ for which~$\stirlingtwo{p+k-1}{k}\not\equiv 1
\pmod{p}$.
By Lemma~\ref{lemmaS2CongruencesAndPeriodicity},
it is enough to find a prime~$p$ with~$k\notin\{1,\dots,p\}\pmod{p^2}$.
This in turn certainly follows from the ({\it{a priori}} stronger)
claim that
\begin{equation}\label{equationUsingBertrand}
\N_{\ge3}=\{n\in\N\mid n\ge3\}
\subset\bigcup_{p\in\PP}(p,p^2],
\end{equation}
since~$k\in(p,p^2]$ precludes~$k\in\{1,\dots,p\}\pmod{p^2}$.
For~$k\ge3$ we therefore seek a prime with~$p<k\le p^2$,
or equivalently with~$\sqrt{k}\le p<k$.
For~$k=3$ we have~$\sqrt{3}\le2<3$.
For~$k\ge4$ we have~$2\le\sqrt{k}\le\frac{k}{2}$,
and by Bertrand's postulate there is a prime~$p$
with~$\frac{k}{2}<p<k$, hence with~$\sqrt{k}\le p<k$,
and so we have~$p<k\le p^2$, which proves~\eqref{equationUsingBertrand}.
\end{proof}

\begin{proof}[Proof of Theorem~\ref{theoremsecondkind} for~$k\ge3$: almost realizability.]
Assume that~$k\ge3$. The statement we wish to prove is
that~$(k-1)!\sequencetwo_k$ is realizable, and we do this
by starting with one of the closed formulas for
the Stirling numbers of the second kind,
\begin{equation*}\label{equationStirlingSecondKindClosedForm1}
\stirlingtwo{n}{k}=\frac{1}{k!}\sum_{j=1}^{k}(-1)^{k-j}\binom{k}{j}j^n
\end{equation*}
for~$n\in\N$ and~$1\le k\le n$.
It follows that
\begin{align*}
(k-1)!\stirlingtwo{n}{k}
=
\frac{1}{k}\sum_{j=1}^{k}(-1)^{k-j}\binom{k}{j}j^n
&=
\sum_{j=1}^{k}(-1)^{k-j}\frac{j}{k}\frac{(k)_j}{j!}j^{n-1}\\
&=
\sum_{j=1}^{k}(-1)^{k-j}\frac{(k-1)_{j-1}}{(j-1)!}j^{n-1}\\
&=
\sum_{j=1}^{k}(-1)^{k-j}\binom{k-1}{j-1}j^{n-1},
\end{align*}
where we have used the notation~$(x)_n$ for the falling factorial
as usual, which means that~$(k-1)!\stirlingtwo{n}{k}$
is an integral linear combination of the
functions~$n\mapsto j^{n-1}$.
Write~$\nu_p(n)$ for the power of~$p$ dividing~$n$, so~$\vert n\vert_p=p^{-\nu_p(n)}$.
We claim that
\begin{equation}\label{equationS2SimplifypAdic1}
\sum_{d\smalldivides n}\mu(d)(k-1)!\stirlingtwo{\tfrac{n}{d}+k-1}{k}
\equiv
0\pmod{p^{\nu_p(n)}}
\end{equation}
for any prime~$p$ dividing~$n$.
To see this, we first write
\[
\sum_{d\smalldivides n}\mu(d)(k-1)!\stirlingtwo{\tfrac{n}{d}+k-1}{k}
=\sum_{\substack{d\smalldivides(n/p)\\p\smallnotdivides d}}
\mu(d)C(n,k,d,p)
\]
where
\[
C(n,k,d,p)
\!=\!
\(\!(k-1)!\stirlingtwo{\tfrac{n}{d}+k-1}{k}
-(k-1)!\stirlingtwo{\tfrac{n}{pd}+k-1}{k}\!\)\!.
\]
Write~$\frac{n}{d}=mp^a$ where~$a=\nu_p(n)$,
for some~$m\in\N$ for~$d$ a divisor of~$\frac{n}{p}$ not
divisible by~$p$.
Then,
\begin{align}
C(n,k,d,p)
&=
\sum_{j=1}^{k}(-1)^{k-j}\binom{k-1}{j-1}\(j^{\frac{n}{d}-1+k-1}-j^{\frac{n}{pd}-1+k-1}\)\nonumber\\
&=
\sum_{j=1}^{k}(-1)^{k-j}\binom{k-1}{j-1}\(j^{mp^a+k-2}-j^{mp^{a-1}+k-2}\)\nonumber\\
&=
\sum_{j=1}^{k}(-1)^{k-j}\binom{k-1}{j-1}j^{mp^{a-1}+k-2}\(j^{mp^{a-1}(p-1)}-1\).\label{equationLastExpressionForC1}
\end{align}
Now~$p^a\divides\(j^{mp^{a-1}(p-1)}-1\)$ by Euler's theorem if~$p\notdivides j$,
and~$p^a\divides j^{p^{a-1}}$ if~$p\divides j$, so~\eqref{equationLastExpressionForC1}
vanishes modulo~$p^a$.
That is, we have proved the claim~\eqref{equationS2SimplifypAdic1}.

By taking the prime decomposition of~$n$, we deduce that
\[
n\divides(\mu*(k-1)!\sequencetwo_k)(n)
\]
for all~$n\ge1$, so~$(k-1)!\sequencetwo_k$ satisfies the Dold condition.
Almost realizability follows by Lemma~\ref{lemmaYash}.
\end{proof}

\section{Combinatorial Proofs}
\label{sectionLemmaProofs}

We assemble here the proofs of the purely combinatorial
lemmas used in the previous section.

\begin{proof}[Proof of Lemma~\ref{lemmaYash}.]
We start with a simple observation from Puri's thesis~\cite{yash}:
if~$(A_n)$ is an increasing sequence of non-negative
real numbers with
\begin{equation}\label{equationGrowthConditionDoubling}
A_{2n}\ge nA_n
\end{equation}
for all~$n\ge1$,
then we claim that
\begin{equation}\label{equationYash1}
(\mu* A)_n\ge0
\end{equation}
for all~$n\ge1$.
To see this, notice that
if~$n$ is even, then we have
\begin{align*}
(\mu*A)_{2n}
=
\sum_{d\smalldivides2n}\mu(2n/d)A_d
\ge
A_{2n}-\sum_{k=1}^{n}A_k
\ge
A_{2n}-nA_n\ge0,
\end{align*}
since the largest proper divisor of~$2n$ is~$n$.
Similarly, in the odd case we have
\begin{align*}
(\mu*A)_{2n+1}
\ge
A_{2n+1}-\sum_{k=1}^{n}A_k\ge A_{2n}-nA_n\ge0,
\end{align*}
since the largest proper divisor of~$2n+1$ is smaller than~$n$,
proving~\eqref{equationYash1}.

For~$k\le2$ we know that~$\stirlingtwo_k$ is
realizable, and so has~$(\mu*\sequencetwo_k)(n)\ge0$ for all~$n\ge1$.

So it is enough to verify that~$\sequencetwo_k$
satisfies~\eqref{equationGrowthConditionDoubling} for~$k\ge3$.
By definition, we know that~$\stirlingtwo{n+k-1}{k}$
is the number of partitions of a set with~$n+k-1$
elements into~$k$ non-empty subsets, and~$\stirlingtwo{2n+k-1}{k}$
is the number of partitions of a set with an additional~$n$
elements into~$k$ non-empty subsets. We can adjoin each of these~$n$ elements to
any one of the subsets in a partition of the
set with~$n+k-1$ elements, showing that
\[
\stirlingtwo{2n+k-1}{k}\ge n\cdot\stirlingtwo{n+k-1}{k}
\]
for all~$n\ge1$, as required.
\end{proof}

\begin{proof}[Proof of Lemma~\ref{lemmaStirlingOneMainCombinatorialArgument}]
Assume that~$n\ge2$, and make use of the
definition of~$\stirlingone{n}{k}$
as the number of permutations of~$\{1,\dots,n\}$
with exactly~$k$ cycles.
Let~$\Sigma$ be a set with~$p+k-1$ elements,
and assume that we have partitioned~$\Sigma$
into~$k$ disjoint non-empty cycles~$\Sigma_1,\Sigma_2,\dots,\Sigma_k$.
We claim that any one of these cycles~$\Sigma_{j}$ has
length no more than~$p$, because each
of the other~$k-1$ subsets in the partition
has at least one element, and so
\[
\vert\Sigma_{j}\vert
=
p+k-1-
\Bigl\vert\bigsqcup_{i\neq j}\Sigma_i\Bigr\vert
\le
p+k-1-(k-1)=p.
\]
For~$\ell=1,\dots,p$ let
\[
j_{\ell}=\vert\{i\mid1\le i\le k,\vert\Sigma_i\vert=\ell\}\vert
\]
denote the number of subsets in the partition with exactly~$\ell$ elements.
It follows that
\begin{equation*}\label{equationStirlingOnePartition1}
\sum_{\ell=1}^{p}j_{\ell}=k
\end{equation*}
since~$\Sigma$
is partitioned into~$k$ non-empty subsets,
and
\begin{equation*}\label{equationStirlingOnePartition2}
\sum_{\ell=1}^{p}\ell j_{\ell}
=\vert\Sigma\vert=p+k-1.
\end{equation*}
In order to count the ways in which~$\Sigma$ can be
partitioned into~$k$ cycles, notice first that
we can begin by partitioning the set into~$j_1$
singletons,~$j_2$ subsets of caridnality~$2$,~$j_3$
of cardinality~$3$, and so on in
\[
\frac{(p+k-1)!}{(1!)^{j_1}j_1!\cdot(2!)^{j_2}j_2!\cdot(3!)^{j_3}j_3!\cdots(p!)^{j_p}j_p!}
\]
ways and then form a cycle from one of the
subsets of cardinality~$j$ in~$(j-1)!$ ways
by fixing the first element in the subset to be the
starting point of the cycle, and then
order the remaining elements in the cycle in any one of~$(j-1)!$
ways.
It follows that we can partition~$\Sigma$ into~$k$
disjoint cycles with~$j_{\ell}$ subsets of cardinality~$\ell$
for~$1\le\ell\le p$ in
\[
\frac{(p+k-1)!}{1^{j_1}j_1!\cdot2^{j_2}j_2!\cdot3^{j_3}j_3!\cdots p^{j_p}j_p!}
\]
ways.
This gives
\begin{equation}\label{equationFormula1ForStirlingOneNumbers}
\stirlingone{p+k-1}{k}
=
\sum_{\substack{(j_1,\dots,j_p)\in\N^p\\
j_1+j_2+\cdots+j_p=k\\
j_1+2j_2+\cdots+pj_p=p+k-1}}
\frac{(p+k-1)!}{1^{j_1}j_1!\cdot2^{j_2}j_2!\cdot3^{j_3}j_3!\cdots p^{j_p}j_p!}.
\end{equation}
We will need some simple counting arguments, as follows:
\begin{enumerate}
\item[(a)] If~$j_{\ell}\ge p$ for some~$\ell\in\{2,\dots,p\}$, then we
have a partition of~$\Sigma$ with at least~$p$ cycles of length
at least~$2$ and~$k-j_{\ell}$ cycles of length at least~$1$.
This requires that~$\vert\Sigma\vert\ge 2j_{\ell}+(k-j_{\ell})\ge k+j_{\ell}\ge k+p$,
a contradiction. So we have~$j_{\ell}<p$ for~$\ell\in\{2,\dots,p\}$.
\item[(b)] If~$j_1=k$, then~$\Sigma$ is partitioned into~$k$
singletons, contradicting~$\vert\Sigma\vert=p+k-1$. Thus~$j_1\le k-1$.
\item[(c)] Finally, if~$j_p>0$, then~$j_p=1$,~$j_{1}=k-1$
and~$j_{\ell}=0$ for~$\ell\in\{2,\dots,p-1\}$.
Indeed, if~$j_p\ge1$, then we have a partition of~$\Sigma$
with~$j_p$ cycles of length~$p$ and~$k-j_{p}$ cycles
of length at least~$1$. This implies that
\[
k+p-1=\vert\Sigma\vert\ge
pj_p+(k-j_p)=k+j_p(p-1),
\]
which is possible only if~$j_p=1$
and all the cycles of length less than~$p$ are singletons.
\end{enumerate}

Now assume that~$(j_1,j_2,\dots,j_{p-1},j_p)\neq(k-1,0,\dots,0,1)$,
so~$j_p=0$ by~(c) and the corresponding term in~\eqref{equationFormula1ForStirlingOneNumbers}
takes the form
\[
\frac{(p+k-1)!}{1^{j_1}j_1!\cdot2^{j_2}j_2!\cdot3^{j_3}j_3!\cdots(p-1)^{j_{p-1}}j_{p-1}!},
\]
in which the values of~$\ell$ and of~$j_{\ell}!$
for~$\ell\in\{2,\dots,p-1\}$ are not divisible
by~$p$, since~$\ell$ and~$j_{\ell}$ are smaller
than~$p$ (the latter by~(a)).
Since~$j_1\le k-1$, we know that~$\vert(p-k+1)!\vert_p<\vert j_1!\vert_p$. It follows that the term in~\eqref{equationFormula1ForStirlingOneNumbers}
corresponding to any~$(j_1,j_2,\dots,j_{p-1},j_p)\neq(k-1,0,\dots,0,1)$
vanishes modulo~$p$.
It follows that~$\stirlingone{p+k-1}{k}$ is congruent modulo~$p$
to the term corresponding to~$(j_1,j_2,\dots,j_{p-1},j_{p})
=(k-1,0,\dots,0,1)$, giving
\begin{equation}\label{equationInMrsInglethorpesRoom1}
\stirlingone{p+k-1}{k}
=
\frac{(p+k-1)!}{(k-1)!p}
=
\frac{1}{p}(p+k-1)_p,
\end{equation}
where we have used the usual
notation~$(x)_n=x(x-1)(x-2)\cdots(x-n+1)$ for the falling
factorial.

We now wish to simplify the expression in~\eqref{equationInMrsInglethorpesRoom1}.
Let~$m=\lceil\frac{k}{p}\rceil$ so that~$k=mp-i$
for some~$i\in\{0,\dots,p-1\}$.
Then
\begin{align*}
\frac{1}{p}(p+k-1)_p
&=
\frac{1}{p}\prod_{j=0}^{p-1}(k+j)\\
&=
\frac{1}{p}\prod_{j=0}^{p-1}(mp-i+j)\\
&=
\frac{1}{p}\left(\prod_{j=p-i}^{p-1}((m-1)p+j)\right)
\cdot mp\cdot\left(\prod_{j=1}^{p-i-1}(mp+j)\right)\\
&=
m\cdot\left(\prod_{j=p-i}^{p-1}j\right)\cdot\left(\prod_{j=1}^{p-i-1}j\right)
\equiv
-m\pmod{p}
\end{align*}
by Wilson's theorem.
We deduce that
\[
\stirlingone{p+k-1}{k}=-\left\lceil\frac{k}{p}\right\rceil.
\]
It follows that in order for the congruence~\eqref{equationStirlingOneValuesModulop1}
to hold, we must have~$m=np+p-1$
for some~$n\in\N$, which is equivalent to
\[
k=mp-i=np^2+(p-1)p-i
\]
with~$i\in\{0,1,\dots,p-1\}$.
This shows the first part of the lemma.

Finally,~\eqref{equationStirlingOneValuesModulop2SignedCase}
follows from the fact
that
\begin{align*}
\stirlingonesigned{p+k-1}{k}
&=
(-1)^{p+k-1-k}\stirlingone{p+k-1}{k}\\
&=
(-1)^{p-1}\stirlingone{p+k-1}{k}\\
&=
\stirlingone{p+k-1}{k}
\end{align*}
for an odd prime~$p$.
\end{proof}

\begin{proof}[Proof of Lemma~\ref{lemmaS2CongruencesAndPeriodicity}]
The proof proceeds in two stages, the first of which is
to find a binomial congruence for~$\sequencetwo_k$.
We let~$\Sigma$ be a set with~$p+k-1$ elements,
and note first that if~$\Sigma$ is partitioned into~$k$
non-empty subsets~$\Sigma_1,\dots,\Sigma_k$,
then
\[
\vert\Sigma_j\vert
=
p+k-1-\Bigl\vert\bigsqcup_{i\neq j}\Sigma_i\Bigr\vert
\le
p+k-1-(k-1)=p
\]
for~$j=1,\dots,k$.
For~$\ell=1,\dots,p$ let
\[
j_{\ell}=\vert\{i\mid1\le i\le k,\vert\Sigma_i\vert=\ell\}\vert
\]
denote the number of subsets in the partition with exactly~$\ell$ elements.
By construction
\begin{equation}\label{equationPartition1}
\sum_{\ell=1}^{p}j_{\ell}=k
\end{equation}
since~$\Sigma$
is partitioned into~$k$ non-empty subsets,
and
\begin{equation}\label{equationPartition2}
\sum_{\ell=1}^{p}\ell j_{\ell}
=\vert\Sigma\vert=p+k-1.
\end{equation}
Since we are not interested in the
arrangement of elements within each subset of the partition,
the number of ways to choose this partition is given by
\[
\frac{(p+k-1)!}{(1!)^{j_1}j_1!\cdot(2!)^{j_2}j_2!\cdot(3!)^{j_3}j_3!\cdots(p!)^{j_p}j_p!}.
\]
Clearly the conditions~\eqref{equationPartition1} and~\eqref{equationPartition2}
are the only requirements for a partition into~$k$ non-empty subsets,
so we can use these conditions to parameterise all such partitions.
Thus we can write
\begin{equation}\label{equationFormula1ForStirlingTwoNumbers}
\stirlingtwo{p+k-1}{k}
=\!\!\!\!
\sum_{\substack{(j_1,\dots,j_p)\in\N^p\\
j_1+j_2+\cdots+j_p=k\\
j_1+2j_2+\cdots+pj_p=p+k-1}}\!\!\!\!
\frac{(p+k-1)!}{(1!)^{j_1}j_1!\!\cdot\!(2!)^{j_2}j_2!\!\cdot\!(3!)^{j_3}j_3!\cdots(p!)^{j_p}j_p!}
\end{equation}
We will need some simple counting arguments, as follows:
\begin{enumerate}
\item[(a)] If~$j_{\ell}\ge p$ for some~$\ell\in\{2,\dots,p\}$,
then there
is a corresponding partition of~$\Sigma$ with at least~$p$
subsets of cardinality at least~$2$ and~$k-p$ subsets
of cardinality at least~$1$. This requires that~$\Sigma$
has cardinality at least
\[
2p+(k-p)>p+k-1,
\]
a contradiction.
So~$j_{\ell}<p$ for all~$\ell\in\{2,\dots,p\}$.
\item[(b)] If~$j_1=k$, then~$\Sigma$ is partitioned
into~$k$ singletons, which requires
\[
\vert\Sigma\vert=k<p+k-1,
\]
a contradiction. So~$j_1\le k-1$.
\item[(c)] If~$j_p\ge1$, then~$\Sigma$ is partitioned
into~$j_p$ subsets of cardinality~$p$
and~$k-j_p$ subsets of cardinality at least~$1$.
This requires
\[
\vert\Sigma\vert
=
p+k-1
\ge
pj_p+(k-j_p)
=
k+j_p(p-1),
\]
which is possible only if~$j_p=1$
and all the subsets of cardinality less than~$p$
are singletons.
\end{enumerate}
These observations allow the expression~\eqref{equationFormula1ForStirlingTwoNumbers}
to be simplified modulo~$p$ as follows.
For a summand with
\[
(j_1,j_2,\dots,j_{p-1},j_p)\neq(k-1,0,\dots,0,1)
\]
we must have~$j_p=0$ by~(c) and the corresponding summand
has the form
\[
\frac{(p+k-1)!}{(1!)^{j_1}j_1!\cdot(2!)^{j_2}j_2!\cdot(3!)^{j_3}j_3!\cdots((p-1)!)^{j_{p-1}}j_{p-1}!},
\]
where for each~$\ell\in\{2,\dots,p-1\}$ the values of~$\ell!$
and~$j_{\ell}!$ are not divisible by~$p$
(since~$\ell,j_{\ell}<p$).
We also have~$j_1\le k-1$ by~(b),
so~$\vert(p+k-1)!\vert_p<\vert j_1!\vert_p$,
and hence any summand corresponding to~$(j_1,j_2,\dots,j_{p-1},j_p)\neq(k-1,0,\dots,0,1)$
vanishes modulo~$p$.
Thus the only term that contributes to~\eqref{equationFormula1ForStirlingTwoNumbers}
modulo~$p$ is the term corresponding to~$(j_1,j_2,\dots,j_{p-1},j_p)=(k-1,0,\dots,0,1)$,
which is
\[
\frac{(p+k-1)!}{(k-1)!p!}=\binom{p+k-1}{p}.
\]
It follows that
\begin{equation}\label{equationBinomialCongruence1}
\stirlingtwo{p+k-1}{k}
\equiv
\binom{p+k-1}{p}\pmod{p}.
\end{equation}
There are well-known ways to express the right-hand side
of~\eqref{equationBinomialCongruence1}. For completeness we
give an argument in the spirit of counting cycles
and the Dold congruence which
lie behind all our results.
Write~$k+p-1$ as~$ap+b$ with~$0\le b<p$
and~$a=\lfloor\frac{k+p-1}{p}\rfloor$
and use these numbers~$a$ and~$b$ to decompose~$\Sigma=\{1,2,\dots,ap+b\}$
into~$a$ groups of~$p$ numbers and one group of~$b$
remaining numbers
(see Figure~\ref{figureA}).

\beginfig
\psfrag{A}{$\overbrace{\hspace{53pt}}^{p}$}
\psfrag{B}{$\overbrace{\hspace{53pt}}^{p}$}
\psfrag{C}{$\overbrace{\hspace{53pt}}^{p}$}
\psfrag{D}{$\overbrace{\hspace{13pt}}^b$}
\psfrag{E}{$\underbrace{\hspace{189pt}}_{ap+b}$}
\scalebox{0.9}{\includegraphics{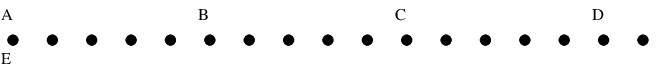}}
\caption{\label{figureA}Defining the cycles in~$\sigma$.}
\endfig

Define a map~$\sigma\colon\Sigma\to\Sigma$
by
\[
\sigma(j)
=
\begin{cases}
j&\mbox{for }j>ap;\\
j+1&\mbox{for }j\le ap,p\notdivides j;\\
j-p+1&\mbox{for }p\divides j.
\end{cases}
\]
The map~$\sigma$ is a permutation of~$\Sigma$
with cycle type given by~$a$ cycles of length~$p$ and~$b$
cycles of length~$1$
(see Figure~\ref{figureB} for an illustration of
this).
Thus~$\sigma^p$ is the identity map.

\beginfig
\psfrag{1}{$1$}
\psfrag{17}{$17$}
\scalebox{0.9}{\includegraphics{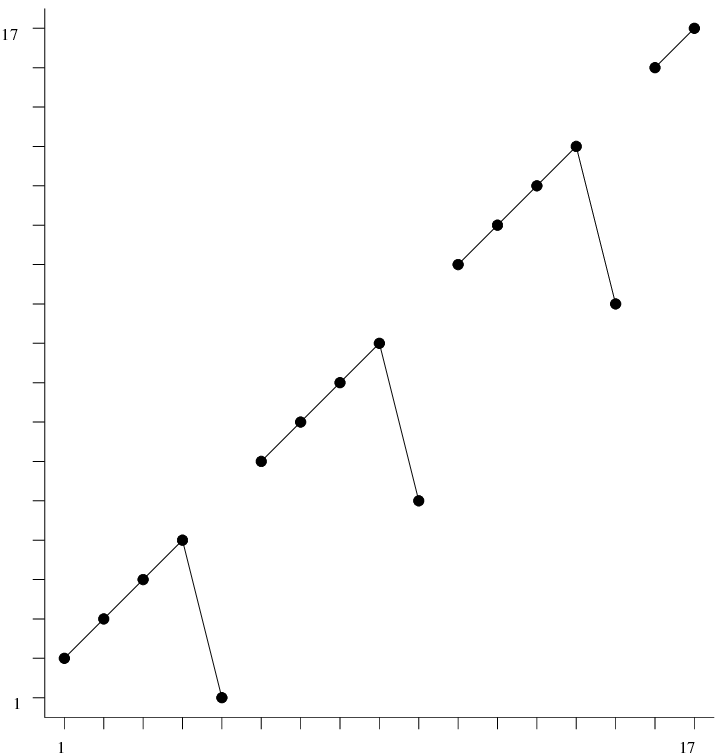}}
\caption{\label{figureB}The permutation~$\sigma$ for~$p=5$,~$a=3$, and~$b=2$.}
\endfig

Now~$\binom{ap+b}{p}$ is the number of
subsets of~$\Sigma$ with cardinality~$p$.
Each of the sets~$\{1,\dots,p\},\{p+1,\dots,2p\},\dots,\{ap-p+1,\dots,ap\}$
are fixed by~$\sigma$, so they each have an orbit
of length~$1$ under the action of~$\sigma$ on
the set of subsets of~$\Sigma$ of cardinality~$p$.
All other subsets have orbits of length~$p$
since~$\sigma$ does not fix them while~$\sigma^p$
is the identity. Thus~$\binom{ap+b}{p}-a$
is divisible by~$p$, and so
\[
\stirlingtwo{p+k-1}{k}
\equiv
\binom{p+k-1}{p}
\equiv
a
=
\left\lceil\frac{k}{p}\right\rceil
\pmod{p}.
\]
This gives~\eqref{equationWhenIsS21ModuloP},
and hence proves the lemma.
\end{proof}

\section{The Sequence of Repair Factors}

Implicit in Definition~\ref{definitionAlmostRealizable} is
the `repair factor'~$\failure(A)$ for an almost realizable
sequence~$A$.
Theorem~\ref{theoremsecondkind} thus gives rise to
a well-defined sequence~$\bigl(\failure(\sequencetwo_k)\bigr)_{k\ge1}$.
We know that~$\failure(\sequencetwo_k)$ divides~$(k-1)!$.
Some calculations are shown in Table~\ref{table}. Here we list hypothetical values of~$\failure(\sequencetwo_k)$ computed as the least common multiples of
denominators of~$\frac{1}{n}(\mu\ast\sequencetwo_k)(n)$ for~$n\in\{1,\ldots,3000\}$,
and then tested as realizable for~$n\leq 50000$. In order to distinguish them from the actual values, we denote them by~$\failure'(\sequencetwo_k)$. We have included columns
showing the prime factorization,
and showing~$\failure'(\sequencetwo_k)/\rad\bigl((k-1)!\bigr)$
(here~$\rad\bigl( m\bigr)$ denotes the greatest square-free divisor,
or radical, of~$m$) as a possible
aid towards trying to answer the natural
question: Is there a closed expression for~$\failure(\sequencetwo_k)$?

\begin{table}[h!]
\begin{center}
\caption{\label{table}Hypothetical values of the repair factor for~$\sequencetwo_k$
in the range $1\le k\le 40$, according to numerical computations.}
\begin{tabular}{c|c|l|c}
$k$&$\failure'(\sequencetwo_k)$&Factorization of $\failure'(\sequencetwo_k)$&$\vphantom{\frac{\sum_A^A}{\sum_A^A}}\frac{\failure'(\sequencetwo_k)}{\rad((k-1)!)}$\\
\hline
1&1&1&1\\
2&1&1&1\\
3&2&$2^1$&1\\
4&6&$2^1\cdot 3^1$&1\\
5&12&$2^2\cdot 3^1$&2\\
6&60&$2^2\cdot 3^1\cdot 5^1$&2\\
7&30&$2^1\cdot 3^1\cdot 5^1$&1\\
8&210&$2^1\cdot 3^1\cdot 5^1\cdot 7^1$&1\\
9&840&$2^3\cdot 3^1\cdot 5^1\cdot 7^1$&4\\
10&2520&$2^3\cdot 3^2\cdot 5^1\cdot 7^1$&12\\
11&1260&$2^2\cdot 3^2\cdot 5^1\cdot 7^1$&6\\
12&13860&$2^2\cdot 3^2\cdot 5^1\cdot 7^1\cdot 11^1$&6\\
13&13860&$2^2\cdot 3^2\cdot 5^1\cdot 7^1\cdot 11^1$&6\\
14&180180&$2^2\cdot 3^2\cdot 5^1\cdot 7^1\cdot 11^1\cdot 13^1$&6\\
15&90090&$2^1\cdot 3^2\cdot 5^1\cdot 7^1\cdot 11^1\cdot 13^1$&3\\
16&30030&$2^1\cdot 3^1\cdot 5^1\cdot 7^1\cdot 11^1\cdot 13^1$&1\\
17&240240&$2^4\cdot 3^1\cdot 5^1\cdot 7^1\cdot 11^1\cdot 13^1$&8\\
18&4084080&$2^4\cdot 3^1\cdot 5^1\cdot 7^1\cdot 11^1\cdot 13^1\cdot 17^1$&8\\
19&6126120&$2^3\cdot 3^2\cdot 5^1\cdot 7^1\cdot 11^1\cdot 13^1\cdot 17^1$&12\\
20&116396280&$2^3\cdot 3^2\cdot 5^1\cdot 7^1\cdot 11^1\cdot 13^1\cdot 17^1\cdot 19^1$&12\\
21&58198140&$2^2\cdot 3^2\cdot 5^1\cdot 7^1\cdot 11^1\cdot 13^1\cdot 17^1\cdot 19^1$&6\\
22&58198140&$2^2\cdot 3^2\cdot 5^1\cdot 7^1\cdot 11^1\cdot 13^1\cdot 17^1\cdot 19^1$&6\\
23&29099070&$2^1\cdot 3^2\cdot 5^1\cdot 7^1\cdot 11^1\cdot 13^1\cdot 17^1\cdot 19^1$&3\\
24&669278610&$2^1\cdot 3^2\cdot 5^1\cdot 7^1\cdot 11^1\cdot 13^1\cdot 17^1\cdot 19^1\cdot 23^1$&3\\
25&892371480&$2^3\cdot 3^1\cdot 5^1\cdot 7^1\cdot 11^1\cdot 13^1\cdot 17^1\cdot 19^1\cdot 23^1$&4\\
26&4461857400&$2^3\cdot 3\cdot 5^2\cdot 7^1\cdot 11^1\cdot 13^1\cdot 17^1\cdot 19^1\cdot 23^1$&20\\
27&2230928700&$2^2\cdot 3^1\cdot 5^2\cdot 7^1\cdot 11^1\cdot 13^1\cdot 17^1\cdot 19^1\cdot 23^1$&10\\
28&20078358300&$2^2\cdot 3^3\cdot 5^2\cdot 7^1\cdot 11^1\cdot 13^1\cdot 17^1\cdot 19^1\cdot 23^1$&90\\
29&20078358300&$2^2\cdot 3^3\cdot 5^2\cdot 7^1\cdot 11^1\cdot 13^1\cdot 17^1\cdot 19^1\cdot 23^1$&90\\
30&582272390700&$2^2\cdot 3^3\cdot 5^2\cdot 7^1\cdot 11^1\cdot 13^1\cdot 17^1\cdot 19^1\cdot 23^1\cdot 29^1$&90\\
31&291136195350&$2^1\cdot 3^3\cdot 5^2\cdot 7^1\cdot 11^1\cdot 13^1\cdot 17^1\cdot 19^1\cdot 23^1\cdot 29^1$&45\\
32&9025222055850&$2^1\cdot 3^3\cdot 5^2\cdot 7^1\cdot 11^1\cdot 13^1\cdot 17^1\cdot 19^1\cdot 23^1\cdot 29^1\cdot 31^1$&45\\
33&144403552893600&$2^5\cdot 3^3\cdot 5^2\cdot 7^1\cdot 11^1\cdot 13^1\cdot 17^1\cdot 19^1\cdot 23^1\cdot 29^1\cdot 31^1$&720\\
34&48134517631200&$2^5\cdot 3^2\cdot 5^2\cdot 7^1\cdot 11^1\cdot 13^1\cdot 17^1\cdot 19^1\cdot 23^1\cdot 29^1\cdot 31^1$&240\\
35&24067258815600&$2^4\cdot 3^2\cdot 5^2\cdot 7^1\cdot 11^1\cdot 13^1\cdot 17^1\cdot 19^1\cdot 23^1\cdot 29^1\cdot 31^1$&120\\
36&24067258815600&$2^4\cdot 3^2\cdot 5^2\cdot 7^1\cdot 11^1\cdot 13^1\cdot 17^1\cdot 19^1\cdot 23^1\cdot 29^1\cdot 31^1$&120\\
37&36100888223400&$2^3\cdot 3^3\cdot 5^2\cdot 7^1\cdot 11^1\cdot 13^1\cdot 17^1\cdot 19^1\cdot 23^1\cdot 29^1\cdot 31^1$&180\\
38&1335732864265800&$2^3\cdot 3^3\cdot 5^2\cdot 7^1\cdot 11^1\cdot 13^1\cdot 17^1\cdot 19^1\cdot 23^1\cdot 29^1\cdot 31^1\cdot 37^1$&180\\
39&333933216066450&$2^1\cdot 3^3\cdot 5^2\cdot 7^1\cdot 11^1\cdot 13^1\cdot 17^1\cdot 19^1\cdot 23^1\cdot 29^1\cdot 31^1\cdot 37^1$&45\\
40&333933216066450&$2^1\cdot 3^3\cdot 5^2\cdot 7^1\cdot 11^1\cdot 13^1\cdot 17^1\cdot 19^1\cdot 23^1\cdot 29^1\cdot 31^1\cdot 37^1$&45
\end{tabular}
\end{center}
\end{table}

In fact a closed formula for~$\failure(\sequencetwo_k)$
seems unlikely to be accessible.
However, based on numerical computations,
four conjectures
concerning the values of~$\failure(\sequencetwo_k)$
emerge.

\begin{conjecture}\label{conj1}
If~$p<k$ is a prime, then~$\vert\failure(\sequencetwo_k)\vert_p<1$.
\end{conjecture}

\begin{conjecture}\label{conj2}
If~$p\in [\sqrt{k},k)$ is a prime, then~$\vert\failure(\sequencetwo_k)\vert_p=\frac{1}{p}$.
\end{conjecture}

\begin{conjecture}\label{conj3}
The repair factor asymptotically involves non-trivial powers
of prime divisors of~$(k-1)!$, in the sense that
\[
\lim_{k\to\infty}\frac{\failure(\sequencetwo_k)}{\rad\bigl((k-1)!\bigr)}=\infty.
\]
\end{conjecture}

\begin{conjecture}\label{conj4}
For each prime~$p$ and positive integer~$j$, the following holds:
\begin{enumerate}
\item[a)] if~$j=1$,
    then~$\failure\bigl(\sequencetwo_{p+1}\bigr)=p\cdot\failure\bigl(\sequencetwo_p\bigr)$;
\item[b)] if~$j>1$, then~$\failure\bigl(\sequencetwo_{p^j+1}\bigr)\mid p^{j-1}\cdot\failure\bigl(\sequencetwo_{p^j}\bigr)$;
\item[c)] if~$j>1$, then~$\bigl\vert\failure\bigl(\sequencetwo_{p^j}\bigr)\bigr\vert_p=p^{-1}$;
\item[d)] we have $\bigl\vert\failure\bigl(\sequencetwo_{p^j+1}\bigr)\bigr\vert_p=p^{-j}$.
\end{enumerate}
\end{conjecture}

The reverse of Conjecture~\ref{conj1} is clearly true. Indeed, if~$p$ is a prime divisor
of~$\failure(\sequencetwo_k)$, then~$p\divides (k-1)!$
and so~$p<k$.
It is easy to show that the statement of Conjecture~\ref{conj1} holds for a
prime~$p\in [\sqrt{k},k)$,\ because
in that case
\[
p<k\le p^2
\]
and so
\[
\left\vert\frac{1}{p}\bigl(S^{(2)}(p+k-1,k)-S^{(2)}(k,k)\bigr)\right\vert_p>1
\]
by Lemma~\ref{lemmaS2CongruencesAndPeriodicity}.
As a result,~$\bigl\vert\failure(\sequencetwo_k)\bigr\vert_p<1$.
We also point out that the statement
of Conjecture~\ref{conj2}
is satisfied for primes~$p\in\bigl(\frac{k-1}{2},k\bigr)$
as
\[
\frac{1}{p}=\vert (k-1)!\vert_p\leq\bigl\vert\failure(\sequencetwo_k)\bigr\vert_p<1.
\]
This also proves the last part of Conjecture~\ref{conj4} for~$j=1$.


\providecommand{\bysame}{\leavevmode\hbox to3em{\hrulefill}\thinspace}
\providecommand{\MR}{\relax\ifhmode\unskip\space\fi MR }
\providecommand{\MRhref}[2]{%
  \href{http://www.ams.org/mathscinet-getitem?mr=#1}{#2}
}
\providecommand{\href}[2]{#2}

\end{document}